\documentclass[reqno,b5paper]{amsart}

\usepackage{amsmath}
\usepackage{amssymb}
\usepackage{amsthm}
\usepackage{enumerate}
\usepackage[mathscr]{eucal}
\usepackage{eqlist}

\theoremstyle{plain}
\newtheorem{thm}{Theorem}[section]
 \newtheorem{lem}[thm]{Lemma}
 
 \newtheorem{cor}[thm]{Corollary}

\theoremstyle{definition}
\newtheorem{rem}{\textnormal{\textbf{Remark}}}

\theoremstyle{remark}

\numberwithin{equation}{section}

\newcommand{\Mn}{M_{n}(\mathcal{R})}
\newcommand{\A}{\textsc{\textbf{a}}}
\newcommand{\B}{\textsc{\textbf{b}}}
\newcommand{\E}{\textsc{\textbf{e}}}
\newcommand{\F}{\textsc{\textbf{f}}}

\begin{document}
\title[Characterizing Jordan derivations of matrix rings]%
{Characterizing Jordan derivations of matrix rings through zero products }
\author[Hoger Ghahramani]%
{Hoger Ghahramani}

\thanks{{\scriptsize
\hskip -0.4 true cm \emph{MSC(2010)}: 16S50; 47B47; 47B49.
\newline \emph{Keywords}: Jordan derivation, Generalized Joradn derivation, Matrix ring.}}

\newcommand{\acr}{\newline\indent}

\address{Department of
Mathematics\acr University of Kurdistan\acr
                    P. O. Box 416\acr
                    Sanandaj\acr
                    Iran}
\email{h.ghahramani@uok.ac.ir; hoger.ghahramani@yahoo.com}

\thanks{}










\begin{abstract}
Let $\Mn$ be the ring of all $n \times n$ matrices over a unital
ring $\mathcal{R}$, let $\mathcal{M}$ be a $2$-torsion free unital
$\Mn$-bimodule and let $D:\Mn\rightarrow \mathcal{M}$ be an
additive map. We prove that if $D(\A)\B+ \A D(\B)+D(\B)\A+ \B
D(\A)=0$ whenever $\A,\B\in \Mn$ are such that $\A\B=\B\A=0$,
then $D(\A)=\delta(\A)+\A D(\textbf{1})$, where
$\delta:\Mn\rightarrow \mathcal{M}$ is a derivation and
$D(\textbf{1})$ lies in the centre of $\mathcal{M}$. It is also
shown that $D$ is a generalized derivation if and only if
$D(\A)\B+ \A D(\B)+D(\B)\A+ \B D(\A)-\A D(\textbf{1})\B-\B
D(\textbf{1})\A=0$ whenever $\A\B=\B\A=0$. We apply this results
to provide that any (generalized) Jordan derivation from $\Mn$
into a $2$-torsion free $\Mn$-bimodule (not necessarily unital)
is a (generalized) derivation. Also, we show that if
$\varphi:\Mn\rightarrow \Mn$ is an additive map satisfying
$\varphi(\A \B+\B \A)=\A\varphi(\B)+\varphi(\B)\A \quad (\A,\B
\in \Mn)$, then $\varphi(\A)=\A\varphi(\textbf{1})$ for all $\A\in
\Mn$, where $\varphi(\textbf{1})$ lies in the centre of $\Mn$. By
applying this result we obtain that every Jordan derivation of
the trivial extension of $\Mn$ by $\Mn$ is a derivation.
\end{abstract}

\maketitle

\section{Introduction}
Throughout this paper all rings are associative. Let
$\mathcal{A}$ be a unital ring and $\mathcal{M}$ be an
$\mathcal{A}$-bimodule. Recall that an additive map
$D:\mathcal{A}\rightarrow \mathcal{M}$ is said to be a
\emph{Jordan derivation} (or \emph{generalized Jordan derivation})
if $D(ab+ba)=D(a)b+aD(b)+D(b)a+bD(a)$ (or
$D(ab+ba)=D(a)b+aD(b)+D(b)a+bD(a)-aD(1)b-bD(1)a$) for all $a,b\in
\mathcal{A}$. It is called a \emph{derivation} (or
\emph{generalized derivation}) if $D(ab)=D(a)b+aD(b)$ (or
$D(ab)=D(a)b+aD(b)-aD(1)b$) for all $a,b\in \mathcal{A}$. Each map
$I_{m}:\mathcal{A}\rightarrow \mathcal{M}$ given by
$I_{m}(a)=am-ma$ ($m\in \mathcal{M}$) is a derivation which will
be called an \emph{inner derivation}. Clearly, each (generalized)
derivation is a (generalized) Jordan derivation. The converse is,
in general, not true.
\begin{rem}\label{R}
let $\mathcal{A}$ be a unital ring, $\mathcal{M}$ be an
$\mathcal{A}$-bimodule and $D:\mathcal{A}\rightarrow \mathcal{M}$
be an additive mapping. Then the following are equivalent:
\begin{enumerate}
\item[(i)] $D$ is a generalized derivation,
\item[(ii)] there is a derivation $\delta:\mathcal{A}\rightarrow
\mathcal{M}$ such that $D(a)=\delta(a)+aD(1)$ for $a\in
\mathcal{A}$.
\end{enumerate}
If (i) holds, define $\delta:\mathcal{A}\rightarrow \mathcal{M}$
by $\delta(a)=D(a)-aD(1)$. It is easily seen that $\delta$ is a
derivation, so $(ii)$ obtain. Conversely, if $(ii)$ holds we have
\begin{equation*}
\begin{split}
D(ab)&=\delta(ab)+abD(1)=\delta(a)b+a\delta(b)+abD(1)\\&
=(D(a)-aD(1))b+a(D(b)-bD(1))+abD(1)\\& =D(a)b+aD(b)-aD(1)b.
\end{split}
\end{equation*}
Thus $D$ is a generalized derivation.
\end{rem}
\par
The question under what conditions a map becomes a (generalized
or Jordan) derivation attracted much attention of mathematicians.
Herstein\cite{Her} proved that every Jordan derivation from a
$2$-torsion free prime ring into itself is a derivation.
Bre$\check{\textrm{s}}$ar \cite{Bre} showed that every Jordan
derivation from a $2$-torsion free semiprime ring into itself is
a derivation. By a classical result of Jacobson and Rickart [6]
every Jordan derivation on a full matrix ring over a $2$-torsion
free unital ring is a derivation and Alizadeh in \cite{Ali}
generalized this result. For more studies concerning Jordan
derivations we refer the reader to \cite{Ben, Fei, Gh, Jac, Li,
Sin, Zh, Zh2} and the references therein. Also, there have been a
number of papers concerning the study of conditions under which
(generalized or Jordan) derivations of rings can be completely
determined by the action on some sets of points \cite{Ala1, Ala2,
Ala3, Che, Gha, Jia, Jing, Jing2, Zhu, Zhu2}.
\par
In this paper, following \cite{Ala3}, we consider the subsequent
condition on an additive map $D$ from a ring $\mathcal{A}$ into
an $\mathcal{A}$-bimodule $\mathcal{M}$:
\[ a, b \in \mathcal{A}, \quad ab=ba=0\Rightarrow
D(a)b+ aD(b)+D(b)a+ bD(a)=0. \quad \quad (\ast)
\]
Our purpose is to investigate whether the condition $(\ast)$
characterizes Jordan derivations. A similar question is concerned
with generalized Jordan derivations. So we consider the following
condition on an additive map $D:\mathcal{A}\rightarrow
\mathcal{M}$ to the context of generalized Jordan derivations,
where $\mathcal{A}$ is unital and $\mathcal{M}$ is unital
$\mathcal{A}$-bimodule:
\[ab=ba=0\Rightarrow
D(a)b+ aD(b)+D(b)a+ bD(a)-aD(1)b-bD(1)a=0. \quad \quad (\ast \ast)
\]
In Section 2 we prove that, in the case when $\mathcal{A}$ is a
full matrix ring $\Mn$ over a unital ring $\mathcal{R}$ and
$\mathcal{M}$ is a $2$-torsion free unital $\Mn$-bimodule,
conditions $(\ast)$ and $(\ast\ast)$ imply that $D$ is of the form
$D(\A)=\delta(\A)+\A D(\textbf{1})$ for each $\A\in \Mn$, where
$\delta: \Mn\rightarrow \mathcal{M}$ is a derivation and
$\textbf{1}$ is the identity matrix. In the case $(*)$ we have
$D(\textbf{1})\in Z(\mathcal{M})$, where $Z(\mathcal{M})$ is the
centre of $\mathcal{M}$. In section 3 our previous results are
applied to characterize (generalized) Jordan derivations from
$\Mn$ into a $2$-torsion free $\Mn$-bimodule $\mathcal{M}$ which
is not necessarily a unital $\Mn$-bimodule. Indeed, we show that
each (generalized) Jordan derivation from $\Mn$ into
$\mathcal{M}$ is a (generalized) derivation. This generalizes the
main result of \cite{Ali}. In section 4 we get some related
results. In particular, by applying results from section 2 we
obtain that if $\varphi:\Mn\rightarrow \Mn$ is an additive map
satisfying $\varphi(\A \B+\B \A)=\A\varphi(\B)+\varphi(\B)\A
\quad (\A,\B \in \Mn)$, then $\varphi(\A)=\A\varphi(\textbf{1})$
for all $\A\in \Mn$, where $\varphi(\textbf{1})\in Z(\Mn)$. As
applications of the above results, we show that every Jordan
derivation of the trivial extension of $\Mn$ by $\Mn$ is a
derivation.
\begin{rem}
Each of the following conditions on an additive map $D:
\mathcal{A}\rightarrow \mathcal{M}$ implies $(\ast)$, which have
been considered by a number of authors (see, for
instance,\cite{Jia, Zha}):
\begin{equation*}
\begin{split}
&a, b \in \mathcal{A}, \quad ab+ba=0\Rightarrow D(a)b+aD(b)+D(b)a+
bD(a)=0.\\&
 a, b \in \mathcal{A}, \quad ab=0\Rightarrow D(a)b+
aD(b)+D(b)a+ bD(a)=D(ab+ba).
\end{split}
\end{equation*}
Therefore, Theorem~\ref{asli} still holds with each of the above
conditions replaced by $(\ast)$.
\end{rem}
The following notations will be used in this paper.
\par
We shall denote the elements of $\Mn$ by bold letters and the
identity matrix by $\textbf{1}$. Also, $\E_{ij}$ for $1\leq i,j
\leq n$ is the matrix unit, $a\E_{ij}$ is the matrix whose
$(ij)$th entry is $a$ and zero elsewhere, where $a\in \mathcal{R}$
and $1\leq i,j \leq n$, and $a_{i,j}$ is the $(ij)$th entry of
$\A \in \Mn$.

\section{Characterizing Jordan derivations through zero products}
From this point up to the last section $\Mn$, for $n\geq 2$, is
the ring of all $n\times n$ matrices over a unital ring
$\mathcal{R}$ and $\mathcal{M}$ is a $2$-torsion free unital
$\Mn$-bimodule. In this section, we discuss the additive maps
from $\Mn$ into $\mathcal{M}$ satisfying $(\ast)$.
\begin{thm}\label{asli}
Let $D:\Mn \rightarrow \mathcal{M}$ be an additive map satisfying
\[ \A,\B \in \Mn, \quad \A\B=\B\A=0\Rightarrow
D(\A)\B+\A D(\B)+D(\B)\A+\B D(\A)=0. \] Then there exist a
derivation $\delta: \Mn \rightarrow \mathcal{M}$ such that
$D(\A)=\delta(\A)+\A D(\textbf{1})$ for each $\A\in \Mn$ and
$D(\textbf{1})\in Z(\mathcal{M})$.
\end{thm}
\begin{proof}
Set $\E=\E_{11}$ and
$\F=\textbf{1}-\E_{11}=\sum_{j=2}^{n}\E_{jj}$. Then $\E$ and $\F$
are nontrivial idempotents such that $\E+\F=\textbf{1}$ and \E \F
=\F \E =0. Let $\textit{\textbf{m}}=\E D(\E)\F-\F D(\E)\E$.
Define $\Delta:\Mn \rightarrow \mathcal{M}$ by
$\Delta(\A)=D(\A)-I_{\textit{\textbf{m}}}(\A)$. Then $\Delta$ is
an additive mapping which satisfies $(\ast)$. Moreover $\E
\Delta(\E)\F=\F \Delta(\E)\E=0$.
\par
We complete the proof by checking some steps.
\\
\textbf{Step 1.} $\Delta(\E\A\E)=\E\Delta(\E\A\E)\E$ and
$\Delta(\F\A\F )=\F\Delta(\F\A\F)\F$ for all $\A \in \Mn$.
\\ \par
Let $\A \in \Mn$. Since $\E(\F \A \F )=(\F \A \F)\E=0$, we have
\begin{equation}\label{1}
\Delta(\E)\F \A \F +\E \Delta(\F \A \F)+\Delta(\F \A \F )\E +\F
\A \F \Delta(\E)=0.
\end{equation}
Multiplying this identity by $\E$ both on the left and on the
right we see that $2\E \Delta(\F \A \F)\E=0$ so $\E
\Delta(\F\A\F)\E=0$. Now, multiplying the Equation\eqref{1} from
the left by $\E$, from the right by $\F$ and by the fact that
$\E\Delta(\E)\F=0$, we get $\E\Delta(\F\A\F)\F=0$. Similarly, from
Equation\eqref{1} and the fact that $\F\Delta(\E)\E=0$, we see
that $\F\Delta(\F\A\F)\E=0$. Therefore, from above equations we
arrive at
\begin{equation*}
\Delta(\F\A\F )=\F\Delta(\F\A\F)\F
\end{equation*}
We have $(\E\A\E)\F=\F(\E\A\E)=0$. Thus
\begin{equation}\label{3}
\Delta(\E\A\E)\F+\E\A\E\Delta(\F)+\Delta(\F)(\E\A\E)+\F\Delta(\E\A\E)=0.
\end{equation}
By $\Delta(\F\A\F )=\F\Delta(\F\A\F)\F$, Equation\eqref{3} and
using similar methods as above we obtain
\begin{equation*}
\Delta(\E\A\E)=\E\Delta(\E\A\E)\E.
\end{equation*}
\textbf{Step 2.} $\Delta(\E\A\F)=\E\Delta(\E\A\F)\F$ for all $\A
\in \Mn$.
\\ \par
Let $\A, \B \in \Mn$. Since $(\E\A\F)(\E\B\F)=(\E\B\F)(\E\A\F)=0$
we have
\begin{equation}\label{5}
\Delta(\E\A\F)\E\B\F+\E\A\F\Delta(\E\B\F)+\Delta(\E\B\F)\E\A\F+\E\B\F\Delta(\E\A\F)=0.
\end{equation}
Multiplying Equation\eqref{5} by $\E$ both on the left and on the
right, we get
\begin{equation}\label{51}
\E\A\F\Delta(\E\B\F)\E+\E\B\F\Delta(\E\A\F)\E=0.
\end{equation}
Similarly, multiplying Equation\eqref{5} by $\F$ both on the left
and on the right, we find
\begin{equation}\label{52}
\F\Delta(\E\A\F)\E\B\F+\F\Delta(\E\B\F)\E\A\F=0.
\end{equation}
We have
$(\E\A\E+\E\A\E\B\F)(\F-\E\B\F)=(\F-\E\B\F)(\E\A\E+\E\A\E\B\F)=0$
and so
\begin{equation}\label{53}
\begin{split}
&\Delta(\E\A\E+\E\A\E\B\F)(\F-\E\B\F)+(\E\A\E+\E\A\E\B\F)\Delta(\F-\E\B\F)\\&+\Delta(\F-\E\B\F)(\E\A\E+\E\A\E\B\F)+(\F-\E\B\F)\Delta(\E\A\E+\E\A\E\B\F)=0.
\end{split}
\end{equation}
Multiplying Equation\eqref{53} by $\E$ both on the left and on the
right and replacing $\A$ by $\E$, from Step 1 and
Equation\eqref{51}, we get $\E\Delta(\E\B\F)\E=0$. Now
multiplying Equation\eqref{53} by $\F$ both on the left and on
the right, by Equation\eqref{52} and a similar arguments as above
we find $\F\Delta(\E\B\F)\F=0$.
\par
Multiplying Equation\eqref{53} by $\F$ on the left and by $\E$ on
the right. By Step 1, we arrive at
\begin{equation}\label{54}
\F\Delta(\E\A\E\B\F)\E=\F\Delta(\E\B\F)\E\A\E.
\end{equation}
For any $\A\in \Mn$ and $2\leq j \leq n$, let $\E_{11}\A
\E_{jj}=a\E_{1j}$. By Equation\eqref{52} we have
\begin{equation*}
\begin{split}
\F\Delta(a\E_{1j})\E_{1j}&=\F\Delta(\E(a\E_{1j})\F)\E\E_{1j}\F
=-\F\Delta(\E\E_{1j}\F)\E(a\E_{1j})\F\\&
=-\F\Delta(\E\E_{1j}\F)\E (a\E_{11})\E\E_{1j}.
\end{split}
\end{equation*}
Also from Equation\eqref{54} we see that
\begin{equation*}
\F\Delta(\E\E_{1j}\F)\E (a\E_{11})\E\E_{1j}= \F\Delta(\E
a\E_{11}\E\E_{1j}\F)\E\E_{1j}=\F\Delta(a\E_{1j})\E_{1j}.
\end{equation*}
 So $\F\Delta(a\E_{1j})\E_{1j}=-\F\Delta(a\E_{1j})\E_{1j}$ and
 hence
$\F\Delta(a\E_{1j})\E_{1j}=0$. Multiplying this identity on the
right by $\E_{j1}$, we get $\F\Delta(a\E_{1j})\E=0$. Therefore
$\F\Delta(\E_{11}\A \E_{jj})\E=\F\Delta(a\E_{1j})\E=0$. So
\[
\F\Delta(\E\A\F)\E=\F\Delta(\sum_{j=2}^{n}\E_{11}\A\E_{jj})\E=\sum_{j=2}^{n}\F\Delta(\E_{11}\A\E_{jj})\E=0.\]
Now from previous equations it follows that
\begin{equation*}
\Delta(\E\A\F)=\E\Delta(\E\A\F)\F.
\end{equation*}
\textbf{Step 3.} $\Delta(\F\A\E)=\F\Delta(\F\A\E)\E$ for all $\A
\in \Mn$.
\\ \par
Let $\A, \B \in \Mn$. Applying $\Delta$ to
$(\F\A\E)(\F\B\E)=(\F\B\E)(\F\A\E)=0$, we get
\begin{equation}\label{61}
\Delta(\F\A\E)\F\B\E+\F\A\E\Delta(\F\B\E)+\Delta(\F\B\E)\F\A\E+\F\B\E\Delta(\F\A\E)=0.
\end{equation}
Multiplying Equation\eqref{61} by $\E$ both on the left and on
the right, we get
\begin{equation}\label{62}
\E\Delta(\F\A\E)\F\B\E+\E\Delta(\F\B\E)\F\A\E=0.
\end{equation}
Similarly, multiplying Equation\eqref{61} by $\F$ both on the
left and on the right, we have
\begin{equation}\label{63}
\F\A\E\Delta(\F\B\E)\F+\F\B\E\Delta(\F\A\E)\F=0.
\end{equation}
We have
$(\F+\F\A\E)(\F\A\E\B\E-\E\B\E)=(\F\A\E\B\E-\E\B\E)(\F+\F\A\E)=0$
and so
\begin{equation}\label{64}
\begin{split}
&\Delta(\F+\F\A\E)(\F\A\E\B\E-\E\B\E)+(\F+\F\A\E)\Delta(\F\A\E\B\E-\E\B\E)\\&+\Delta(\F\A\E\B\E-\E\B\E)(\F+\F\A\E)+(\F\A\E\B\E-\E\B\E)\Delta(\F+\F\A\E)=0.
\end{split}
\end{equation}
Multiplying Equation\eqref{64} by $\E$ both on the left and on the
right and replacing $\B$ by $\E$, from Step 1 and
Equation\eqref{62}, we get $\E\Delta(\F\A\E)\E=0$. Now
multiplying Equation\eqref{64} by $\F$ both on the left and on
the right, by Equation\eqref{63} and a similar arguments as above
we find $\F\Delta(\F\A\E)\F=0$.
\par
Multiplying Equation\eqref{64} by $\E$ on the left and by $\F$ on
the right. By Step 1, we arrive at
\begin{equation}\label{65}
\E\Delta(\F\A\E\B\E)\F=\E\B\E\Delta(\F\A\E)\F.
\end{equation}
For any $\A\in \Mn$ and $2\leq j \leq n$, let $\E_{jj}\A
\E_{11}=a\E_{j1}$. By Equation\eqref{63} we have
\begin{equation*}
\begin{split}
\E_{j1}\Delta(a\E_{j1})\F&=\F\E_{j1}\E\Delta(\F(a\E_{j1})\E)\F
=-\F(a\E_{j1})\E\Delta(\F\E_{j1}\E)\F\\&
=-\E_{j1}\E(a\E_{11})\E\Delta(\F\E_{j1}\E)\F.
\end{split}
\end{equation*}
Also from Equation\eqref{65} we see that
\begin{equation*}
\E_{j1}\E(a\E_{11})\E\Delta(\F\E_{j1}\E)\F =
\E_{j1}\E\Delta(\F\E_{j1}\E(a\E_{11})\E)\F=\E_{j1}\Delta(a\E_{j1})\F.
\end{equation*}
 So $\E_{j1}\Delta(a\E_{j1})\F=-\E_{j1}\Delta(a\E_{j1})\F$ and
 hence
$\E_{j1}\Delta(a\E_{j1})\F=0$. Therefore
$$\E\Delta(\E_{jj}\A\E_{11})\F=\E\Delta(a\E_{j1})\F=0.$$
So
\[
\E\Delta(\F\A\E)\F=\E\Delta(\sum_{j=2}^{n}\E_{jj}\A\E_{11})\F=\sum_{j=2}^{n}\E\Delta(\E_{jj}\A\E_{11})\F=0.\]
Now from previous equations it follows that
\begin{equation*}
\Delta(\F\A\E)=\F\Delta(\F\A\E)\E.
\end{equation*}
\textbf{Step 4.}
\par
$\E\Delta(\E\A\E\B\F)\F=\E\A\E\Delta(\E\B\F)\F+\E\Delta(\E\A\E)\E\B\F-\E\A\E\B\F\Delta(\F)\F$
\\
and
\par
$\E\Delta(\E\A\F\B\F)\F=\E\Delta(\E\A\F)\F\B\F+\E\A\F\Delta(\F\B\F)\F-\E\A\F\Delta(\F)\F\B\F$
\\
for all $\A, \B \in \Mn$.
\\ \par
Let $\A, \B \in \Mn$. Multiplying Equation\eqref{53} by $\E$ on
the left and by $\F$ on the right, from Step 1 and 2 we obtain
\begin{equation*}
\E\Delta(\E\A\E\B\F)\F=\E\A\E\Delta(\E\B\F)\F+\E\Delta(\E\A\E)\E\B\F-\E\A\E\B\F\Delta(\F)\F.
\end{equation*}
Replacing $\A$ by $\E$ in above equation, we get
\begin{equation}\label{71}
\E\Delta(\E)\E\B\F=\E\B\F\Delta(\F)\F
\end{equation}
Since
$(\E+\E\A\F)(\F\B\F-\E\A\F\B\F)=(\F\B\F-\E\A\F\B\F)(\E+\E\A\F)=0$,
we have
\begin{equation*}
\begin{split}
&\Delta(\E+\E\A\F)(\F\B\F-\E\A\F\B\F)+(\E+\E\A\F)\Delta(\F\B\F-\E\A\F\B\F)\\&+\Delta(\F\B\F-\E\A\F\B\F)(\E+\E\A\F)+(\F\B\F-\E\A\F\B\F)\Delta(\E+\E\A\F)=0
\end{split}
\end{equation*}
Multiplying this identity by $\E$ on the left and by $\F$ on the
right, from Equation\eqref{71} and Step 1 and 2 we arrive at
\begin{equation*}
\E\Delta(\E\A\F\B\F)\F=\E\Delta(\E\A\F)\F\B\F+\E\A\F\Delta(\F\B\F)\F-\E\A\F\Delta(\F)\F\B\F.
\end{equation*}
\textbf{Step 5.}
\par
$\F\Delta(\F\A\E\B\E)\E=\F\Delta(\F\A\E)\E\B\E+\F\A\E\Delta(\E\B\E)\E-\F\Delta(\F)\F\A\E\B\E$
\\ and
\par
$\F\Delta(\F\A\F\B\E)\E=\F\A\F\Delta(\F\B\E)\E+\F\Delta(\F\A\F)\F\B\E-\F\A\F\Delta(\F)\F\B\E$
\\
for all $\A, \B \in \Mn$.
\\ \par
Let $\A, \B \in \Mn$. Multiplying Equation\eqref{64} by $\F$ on
the left and by $\E$ on the right, from Step 1 and 3 we obtain
\begin{equation*}
\F\Delta(\F\A\E\B\E)\E=\F\Delta(\F\A\E)\E\B\E+\F\A\E\Delta(\E\B\E)\E-\F\Delta(\F)\F\A\E\B\E
\end{equation*}
Replacing $\B$ by $\E$ in above equation, we get
\begin{equation}\label{81}
\F\A\E\Delta(\E)\E=\F\Delta(\F)\F\A\E
\end{equation}
Since
$(\E-\F\B\E)(\F\A\F\B\E+\F\A\F)=(\F\A\F\B\E+\F\A\F)(\E-\F\B\E)=0$,
we have
\begin{equation*}\label{82}
\begin{split}
&\Delta(\E-\F\B\E)(\F\A\F\B\E+\F\A\F)+(\E-\F\B\E)\Delta(\F\A\F\B\E+\F\A\F)\\&+\Delta(\F\A\F\B\E+\F\A\F)(\E-\F\B\E)+(\F\A\F\B\E+\F\A\F)\Delta(\E-\F\B\E)=0
\end{split}
\end{equation*}
Multiplying this identity by $\F$ on the left and by $\E$ on the
right, from Equation\eqref{81} and Step 1 and 3 we arrive at
\begin{equation*}
\F\Delta(\F\A\F\B\E)\E=\F\A\F\Delta(\F\B\E)\E+\F\Delta(\F\A\F)\F\B\E-\F\A\F\Delta(\F)\F\B\E.
\end{equation*}
\textbf{Step 6.}
\par
$\E\Delta(\E\A\E\B\E)\E=\E\A\E\Delta(\E\B\E)\E+\E\Delta(\E\A\E)\E\B\E-\E\A\E\Delta(\E)\E\B\E$
\\ and
\par
$\F\Delta(\F\A\F\B\F)\F=\F\Delta(\F\A\F)\F\B\F+\F\A\F\Delta(\F\B\F)\F-\F\A\F\Delta(\F)\F\B\F$
\\
for all $\A, \B \in \Mn$.
\\ \par
Let $\A, \B\in \Mn$. For $2\leq j \leq n$, we have
$\E_{1j}=\E\E_{1j}\F $, so from Step 4 we see that
\[
\E\Delta(\E\A\E\B\E_{1j})\F=\E\A\E\B\E\Delta(\E_{1j})\F+\E\Delta(\E\A\E\B\E)\E_{1j}-\E\A\E\B\E_{1j}\Delta(\F)\F.\]
On the other hand,
\begin{equation*}
\begin{split}
\E\Delta(\E\A\E\B\E_{1j})\F&=
\E\A\E\Delta(\E\B\E_{1j})\F+\E\Delta(\E\A\E)\E\B\E_{1j}-\E\A\E\B\E_{1j}\Delta(\F)\F\\&
=\E\A\E\B\E\Delta(\E_{1j})\F+\E\A\E\Delta(\E\B\E)\E_{1j}
-\E\A\E\B\E_{1j}\Delta(\F)\F\\&+\E\Delta(\E\A\E)\E\B\E_{1j}-\E\A\E\B\E_{1j}\Delta(\F)\F.
\end{split}
\end{equation*}
By comparing the two expressions for
$\E\Delta(\E\A\E\B\E_{1j})\F$, Equation\eqref{71} and multiplying
the resulting equation by $\E_{j1}$ on the right, yields
\begin{equation*}
\E\Delta(\E\A\E\B\E)\E=\E\A\E\Delta(\E\B\E)\E+\E\Delta(\E\A\E)\E\B\E-\E\A\E\Delta(\E)\E\B\E.
\end{equation*}
We have $\E_{1j}=\E\E_{1j}\F$ for $2\leq j \leq n$, so from Step
5 and a proof similar to above, we find
\begin{equation*}
\E_{1j}\Delta(\F\A\F\B\F)\F=\E_{1j}\Delta(\F\A\F)\F\B\F+\E_{1j}\A\F\Delta(\F\B\F)\F-\E_{1j}\A\F\Delta(\F)\F\B\F.
\end{equation*}
Multiplying this identity from left by $\E_{j1}$ we get
\begin{equation*}
\E_{jj}\Delta(\F\A\F\B\F)\F=\E_{jj}\Delta(\F\A\F)\F\B\F+\E_{jj}\A\F\Delta(\F\B\F)\F-\E_{jj}\A\F\Delta(\F)\F\B\F.
\end{equation*}
So
\begin{equation*}
\begin{split}
\F\Delta(\F\A\F\B\F)\F&=\sum_{j=2}^{n}\E_{jj}\Delta(\F\A\F\B\F)\F\\&
=\sum_{j=2}^{n}(\E_{jj}\Delta(\F\A\F)\F\B\F+\E_{jj}\A\F\Delta(\F\B\F)\F-\E_{jj}\A\F\Delta(\F)\F\B\F\\&
=\F\Delta(\F\A\F)\F\B\F+\F\A\F\Delta(\F\B\F)\F-\F\A\F\Delta(\F)\F\B\F.
\end{split}
\end{equation*}
\textbf{Step 7.} $\A\Delta(\textbf{1})=\Delta(\textbf{1})\A$ for
all $\A \in \Mn$.
\\ \par
Let $\A \in \Mn$. By Equation\eqref{71} we have
\begin{equation*}
\begin{split}
&\E\A\E\Delta(\E)\E_{1j}=\E\A\E_{1j}\Delta(\F)\F=\E\Delta(\E)\E\A\E_{1j}
\\& \textrm{and}
\\&
\E_{1j}\Delta(\F)\F\A\F=\E\Delta(\E)\E_{1j}\A\F=\E_{1j}\A\F\Delta(\F)\F
\end{split}
\end{equation*}
for $2\leq j \leq n$. So
\begin{equation}\label{101}
\begin{split}
&\E\A\E\Delta(\E)\E=\E\Delta(\E)\E\A\E, \quad
\E_{jj}\Delta(\F)\F\A\F=\E_{jj}\A\F\Delta(\F)\F
\\& \textrm{and}\,\, \textrm{hence}
\\&
\F\Delta(\F)\F\A\F=\sum_{j=2}^{n}\E_{jj}\Delta(\F)\F\A\F=\sum_{j=2}^{n}\E_{jj}\A\F\Delta(\F)\F=\F\A\F\Delta(\F)\F.
\end{split}
\end{equation}
By Step 1 we have
$\Delta(\textbf{1})=\E\Delta(\E)\E+\F\Delta(\F)\F$. From this
identity and Equations\eqref{71}, \eqref{81}, \eqref{101} we
arrive at
\begin{equation*}
\begin{split}
\A\Delta(\textbf{1})&=\E\A\E\Delta(\textbf{1})+\E\A\F\Delta(\textbf{1})+\F\A\E\Delta(\textbf{1})+\F\A\F\Delta(\textbf{1})
\\&=\E\A\E\Delta(\E)\E+\E\A\F\Delta(\F)\F+\F\A\E\Delta(\E)\E+\F\A\F\Delta(\F)\F
\\&=\E\Delta(\E)\E\A\E+\E\Delta(\E)\E\A\F+\F\Delta(\F)\F\A\E+\F\Delta(\F)\F\A\F
\\&=\Delta(\textbf{1})\E\A\E+\Delta(\textbf{1})\E\A\F+\Delta(\textbf{1})\F\A\E+\Delta(\textbf{1})\F\A\F
\\&=\Delta(\textbf{1})\A.
\end{split}
\end{equation*}
\textbf{Step 8.}
\par
$\E\Delta(\E\A\F\B\E)\E=\E\Delta(\E\A\F)\F\B\E+\E\A\F\Delta(\F\B\E)\E-\E\A\F\B\E\Delta(\E)\E$
\\ and
\par
$\F\Delta(\F\B\E\A\F)\F=\F\Delta(\F\B\E)\E\A\F+\F\B\E\Delta(\E\A\F)\F-\F\Delta(\F)\F\B\E\A\F$
\\
for all $\A, \B \in \Mn$.
\\ \par
Let $\A, \B\in \Mn$. By applying $\Delta$ to
\begin{equation}\label{201}
\begin{split}
(&\E\A\F\B\E+\E\A\F-\F\B\E-\F)(-\E-\E\A\F+\F\B\E+\F\B\E\A\F)\\&=(-\E-\E\A\F+\F\B\E+\F\B\E\A\F)(\E\A\F\B\E+\E\A\F-\F\B\E-\F)=0
\end{split}
\end{equation}
 and multiplying the resulting equation by $\E$ both on the left and on the
right, from Steps 1--3 and Equations\eqref{101} we deduce that
$$\E\Delta(\E\A\F\B\E)\E=\E\Delta(\E\A\F)\F\B\E+\E\A\F\Delta(\F\B\E)\E-\E\A\F\B\E\Delta(\E)\E.$$
Also by applying $\Delta$ to \eqref{201} and multiplying the
resulting equation by $\F$ both on the left and on the right, from
Steps 1--3 and Equations\eqref{101} we get
$$\F\Delta(\F\B\E\A\F)\F=\F\Delta(\F\B\E)\E\A\F+\F\B\E\Delta(\E\A\F)\F-\F\Delta(\F)\F\B\E\A\F.$$
\par We have $D(\textbf{1})=\Delta(\textbf{1})$ and hence from Step
7 we find that $D(\textbf{1})\in Z(\mathcal{M})$. Since
$\A\B=\E\A\E\B\E+\E\A\E\B\F+\E\A\F\B\E+\E\A\F\B\F+\F\A\E\B\E+\F\A\E\B\F+\F\A\F\B\E+\F\A\F\B\F$
for any $\A,\B \in \Mn$, by Steps 1--8, it follows that the
mapping $d:\Mn\rightarrow \mathcal{M}$ given by
$d(\A)=\Delta(\A)-\A\Delta(\textbf{1})$ is a derivation. So the
mapping $\delta:\Mn\rightarrow \mathcal{M}$ given by
$\delta(\A)=d(\A)+I_{\textit{\textbf{m}}}(\A)$ is a derivation
and we have $D(\A)=\delta(\A)+\A D(\textbf{1})$ for all $\A \in
\Mn$. The proof is thus completed.
\end{proof}
The following theorem is a consequence of Theorem~\ref{asli}.
\begin{thm}\label{natije}
Let $D:\Mn \rightarrow \mathcal{M}$ be an additive map satisfying
\[
\A\B=\B\A=0\Rightarrow D(\A)\B+\A D(\B)+D(\B)\A+\B D(\A)-\A
D(\textbf{1})\B-\B D(\textbf{1})\A=0.
\]
Then there exist a derivation $\delta: \Mn \rightarrow
\mathcal{M}$ such that $D(\A)=\delta(\A)+\A D(\textbf{1})$ for
each $\A\in \Mn$.
\end{thm}
\begin{proof}
Define $\delta:\Mn\rightarrow \mathcal{M}$ by $\delta(\A)=D(\A)-\A
D(\textbf{1})$. It is easy too see that $\delta$ is an additive
map satisfying $(\ast)$ and $\delta(\textbf{1})=0$. By
Theorem~\ref{asli}, $\delta$ is a derivation. Thus
$D(\A)=\delta(\A)+\A D(\textbf{1})$ for all $\A\in \Mn$ and proof
is completed.
\end{proof}
Let $\mathcal{R}$ be a unital ring and $\mathcal{N}$ be a unital
$\mathcal{R}$-bimodule. Let $M_{n}(\mathcal{N})$ be the set of all
$n\times n$ matrices over $\mathcal{N}$, then $M_{n}(\mathcal{N})$
has a natural structure as unital $\Mn$-bimodule. Any derivation
$d:\mathcal{R}\rightarrow \mathcal{N}$, induces a derivation
$\bar{d}:\Mn\rightarrow M_{n}(\mathcal{N})$ given by
$\bar{d}(\A)=\textsc{\textbf{n}}$, where $n_{i,j}=d(a_{i,j})$. By
similar method as in proof of \cite[Theorem 3.1]{Ali}, we can
show that if $\delta:\Mn\rightarrow M_{n}(\mathcal{N})$ is a
derivation, then there is an inner derivation
$I_{\textsc{\textbf{g}}}:\Mn\rightarrow M_{n}(\mathcal{N})$ and a
derivation $d:\mathcal{R}\rightarrow \mathcal{N}$ such that
$\delta=\bar{d}+I_{\textsc{\textbf{g}}}$. So by
Theorem~\ref{asli}, we have the following corollary.
\begin{cor}\label{natije2}
Let $\mathcal{R}$ be a unital ring and $\mathcal{N}$ be a
$2$-torsion free unital $\mathcal{R}$-bimodule. Let $D:\Mn
\rightarrow M_{n}(\mathcal{N})$ be an additive mapping.
\begin{itemize}
\item[(i)] If $D$ satisfies $(\ast)$, then there is an inner derivation
$I_{\textsc{\textbf{g}}}:\Mn\rightarrow M_{n}(\mathcal{N})$ and a
derivation $d:\mathcal{R}\rightarrow \mathcal{N}$ such that
$D(\A)=\bar{d}(\A)+I_{\textsc{\textbf{g}}}(\A)+\A D(\textbf{1})$
for all $\A\in \Mn$, where $D(\textbf{1})\in
Z(M_{n}(\mathcal{N})).$
\item[(ii)] If $D$ satisfies $(\ast\ast)$, then there is an inner derivation
$I_{\textsc{\textbf{g}}}:\Mn\rightarrow M_{n}(\mathcal{N})$ and a
derivation $d:\mathcal{R}\rightarrow \mathcal{N}$ such that
$D(\A)=\bar{d}(\A)+I_{\textsc{\textbf{g}}}(\A)+\A D(\textbf{1})$
for all $\A\in \Mn$.
\end{itemize}
\end{cor}
\section{Jordan derivations of matrix rings}
In this section we characterize Jordan derivations of matrix rings
into bimodules which are not necessarily unital bimodule. To
prove the main result, we need the following lemma.
\begin{lem}\label{L}
Let $\mathcal{A}$ be a unital ring. Then the following are
equivalent:
\begin{itemize}
\item[(i)] for every $2$-torsion free unital
$\mathcal{A}$-bimodule $\mathcal{M}$, each Jordan derivation
$D:\mathcal{A}\rightarrow \mathcal{M}$ is a derivation.
\item[(ii)] for every $2$-torsion free $\mathcal{A}$-bimodule $\mathcal{M}$, each Jordan derivation
$D:\mathcal{A}\rightarrow \mathcal{M}$ is a derivation.
\item[(iii)] for every $2$-torsion free $\mathcal{A}$-bimodule $\mathcal{M}$, each generalized Jordan derivation
$D:\mathcal{A}\rightarrow \mathcal{M}$ is a generalized
derivation.
\end{itemize}
\end{lem}
\begin{proof}
$(i)\Rightarrow (ii)$ Let $\mathcal{M}$ be a $2$-torsion free
$\mathcal{A}$-bimodule and $1$ be the unity of $\mathcal{A}$.
Define the following sets:
\begin{equation*}
\begin{split}
&\mathcal{M}_{1}=\{1m1\,|\, m\in \mathcal{M} \},\\&
\mathcal{M}_{2}=\{1m-1m1\,|\, m\in \mathcal{M} \},\\&
\mathcal{M}_{3}=\{m1-1m1\,|\, m\in \mathcal{M} \}\quad and\\&
\mathcal{M}_{4}=\{m-1m-m1+1m1\,|\, m\in \mathcal{M} \}.
\end{split}
\end{equation*}
Every $\mathcal{M}_{j}$ for $1\leq j \leq 4$ is an
$\mathcal{A}$-subbimodule of $\mathcal{M}$ such that
$\mathcal{M}_{1}$ is unital and
\[
\mathcal{M}_{2}\mathcal{A}=\mathcal{A}\mathcal{M}_{3}=\mathcal{M}_{4}\mathcal{A}=\mathcal{A}\mathcal{M}_{4}=\{0\}.\]
Also $1m_{2}=m_{2}$ for all $m_{2}\in \mathcal{M}_{2}$,
$m_{3}1=m_{3}$ for all $m_{3}\in \mathcal{M}_{3}$ and
$\mathcal{M}=\mathcal{M}_{1}\dot{+}\mathcal{M}_{2}\dot{+}\mathcal{M}_{3}\dot{+}\mathcal{M}_{4}$
as sum of $\mathcal{A}$-bimodules. Let $D:\mathcal{A}\rightarrow
\mathcal{M}$ be a Jordan derivation. So
$D=D_{1}+D_{2}+D_{3}+D_{4}$, where each $D_{j}$ is an additive
map from $\mathcal{A}$ to $\mathcal{M}_{j}$. Since
$D(ab+ba)=D(a)b+aD(b)+D(b)a+bD(a)$ for all $a,b\in \mathcal{A}$,
from the above results we get
\begin{equation*}
\begin{split}
D_{1}&(ab+ba)+D_{2}(ab+ba)+D_{3}(ab+ba)+D_{4}(ab+ba)\\&
=D_{1}(a)b+D_{3}(a)b+aD_{1}(b)+aD_{2}(b)+D_{1}(b)a+D_{3}(b)a+bD_{1}(a)+bD_{2}(a).
\end{split}
\end{equation*}
Therefore
\begin{equation}\label{15}
\begin{split}
&D_{1}(ab+ba)=D_{1}(a)b+aD_{1}(b)+D_{1}(b)a+bD_{1}(a),\\&
D_{2}(ab+ba)=aD_{2}(b)+bD_{2}(a),\\&
D_{3}(ab+ba)=D_{3}(a)b+D_{3}(b)a \quad and\\&
D_{4}(ab+ba)=0
\end{split}
\end{equation}
So $D_{1}$ is a Jordan derivation and by hypothesis it is a
derivation since $\mathcal{M}_{1}$ is a $2$-torsion free unital
$\mathcal{A}$-bimodule. Now taking $b=1$ in Equations\eqref{15},
we arrive at $D_{2}(a)=aD_{2}(1)$, $D_{3}(a)=D_{3}(1)a$ and
$2D_{4}(a)=0$. Hence $D_{4}(a)=0$ since $\mathcal{M}$ is
$2$-torsion free. By previous results it is obvious that $D$ is a
derivation.
\par
 $(ii)\Rightarrow (iii)$ Let $\mathcal{M}$ be a
$2$-torsion free $\mathcal{A}$-bimodule and
$D:\mathcal{A}\rightarrow \mathcal{M}$ be a generalized Jordan
derivation. The mapping $\delta:\mathcal{A}\rightarrow
\mathcal{M}$ defined by $\delta(a)=D(a)-aD(1)$ is a Jordan
derivation and hence it is a derivation. So from Remark~\ref{R},
$D$ is a generalized derivation.
\par
 $(iii)\Rightarrow (i)$
Let $\mathcal{M}$ be a $2$-torsion free unital
$\mathcal{A}$-bimodule and $D:\mathcal{A}\rightarrow \mathcal{M}$
be a Jordan derivation. So $D(1)=0$ since $\mathcal{M}$ is a
unital $\mathcal{A}$-bimodule. Hence from hypothesis it is clear
that $D$ is a derivation.
\end{proof}
If $\mathcal{M}$ is a $2$-torsion free unital $\Mn$-bimodule and
$D:\Mn\rightarrow \mathcal{M}$ is a Jordan derivation, then $D$
satisfies $(\ast)$ and $D(\textbf{1})=0$ , and hence $D$ is a
derivation by Theorem~\ref{asli}. So from Lemma~\ref{L} we have
the following theorem which is a generalization of \cite[Theorem
3.1]{Ali}.
\begin{thm}
Let $\mathcal{M}$ be a $2$-torsion free $\Mn$-bimodule and $D:\Mn
\rightarrow \mathcal{M}$ be an additive mapping.
\begin{enumerate}
\item[(i)] If $D$ is a Jordan derivation, then $D$ is a derivation.
\item[(ii)] If $D$ is a generalized Jordan derivation, then $D$ is a generalized derivation.
\end{enumerate}
\end{thm}
By Corollary~\ref{natije2}, the following corollary is obvious.
\begin{cor}
Let $\mathcal{R}$ be a unital ring and let $\mathcal{N}$ be a
$2$-torsion free unital $\mathcal{R}$-bimodule. Let $D:\Mn
\rightarrow M_{n}(\mathcal{N})$ be an additive mapping.
\begin{enumerate}
\item[(i)] If $D$ is a Jordan derivation, then there is an inner derivation
$I_{\textsc{\textbf{g}}}:\Mn\rightarrow M_{n}(\mathcal{N})$ and a
derivation $d:\mathcal{R}\rightarrow \mathcal{N}$ such that
$D(\A)=\bar{d}(\A)+I_{\textsc{\textbf{g}}}(\A)$ for all $\A\in
\Mn$.
\item[(ii)] If $D$ is a generalized Jordan derivation, then there is an inner derivation
$I_{\textsc{\textbf{g}}}:\Mn\rightarrow M_{n}(\mathcal{N})$ and a
derivation $d:\mathcal{R}\rightarrow \mathcal{N}$ such that
$D(\A)=\bar{d}(\A)+I_{\textsc{\textbf{g}}}(\A)+\A D(\textbf{1})$
for all $\A\in \Mn$.
\end{enumerate}
\end{cor}
\section{Some related results}
In this section, by applying results in section 2, we obtain some
results about matrix ring $\Mn$.
\begin{lem}\label{pj}
Let $\mathcal{A}$ be a $2$-torsion free unital ring. Suppose that
each additive mapping $D:\mathcal{A}\rightarrow \mathcal{A}$
satisfying $(\ast)$ is a generalized derivation with $D(1)\in
Z(\mathcal{A})$. Let $\varphi:\mathcal{A}\rightarrow \mathcal{A}$
be an additive map satisfying
\[ \varphi(ab+ba)=a\varphi(b)+\varphi(b)a \quad (a,b\in \mathcal{A}). \]
Then $\varphi(a)=a\varphi(1)$ for all $a\in \mathcal{A}$, where
$\varphi(1)\in Z(\mathcal{A})$.
\end{lem}
\begin{proof}
Let $a,b\in \mathcal{A}$ with $ab=ba=0$. So $ab+ba=0$ and hence
\[ \varphi(ab+ba)=a\varphi(b)+\varphi(b)a=0, \]
\[ \varphi(ba+ab)=b\varphi(a)+\varphi(a)b=0. \]
Therefore, $a\varphi(b)+\varphi(b)a+b\varphi(a)+\varphi(a)b=0$
and $\varphi$ satisfies $(\ast)$. Thus by hypothesis $\varphi$ is
a generalized derivation with $\varphi(1)\in Z(\mathcal{A})$. So
we have
\begin{equation*}
\begin{split}
& \varphi(ab)=a\varphi(b)+\varphi(a)b-a\varphi(1)b\\& and
\\&
\varphi(ba)=b\varphi(a)+\varphi(b)a-b\varphi(1)a.
\end{split}
\end{equation*}
for all $a,b\in \mathcal{A}$. From these identities we get
\[ \varphi(ab+ba)= 2\varphi(ab+ba)-ab\varphi(1)-ba\varphi(1). \]
Hence $\varphi(ab+ba)=ab\varphi(1)+ba\varphi(1)$. Letting $b=1$ in
this identity, we find $\varphi(a)=a\varphi(1)$ for all $a\in
\mathcal{A}$, where $\varphi(1)\in Z(\mathcal{A})$.
\end{proof}
By Theorem~\ref{asli} and Lemma~\ref{pj} we have the following
theorem.
\begin{thm}
Let $\mathcal{R}$ be a $2$-torsion free unital ring, and let
$\varphi:\Mn \rightarrow \Mn$ be an additive map satisfying
\[ \varphi(\A \B+\B \A)=\A\varphi(\B)+\varphi(\B)\A \quad (\A,\B \in \Mn). \]
Then $\varphi(\A)=\A\varphi(\textbf{1})$ for all $\A\in \Mn$,
where $\varphi(\textbf{1})\in Z(\Mn)$.
\end{thm}
Given a ring $\mathcal{A}$ and an $\mathcal{A}$-bimodule
$\mathcal{M}$, the \emph{trivial extension} of $\mathcal{A}$ by
$\mathcal{M}$ is the ring
$T(\mathcal{A},\mathcal{M})=\mathcal{A}\oplus\mathcal{M}$ with
the usual addition and the following multiplication:
\[ (a_{1},m_{1})(a_{2},m_{2})=(a_{1}a_{2}, a_{1}m_{2}+m_{1}a_{2}).
\]
\begin{lem}\label{ptri}
Let $\mathcal{A}$ be a $2$-torsion free unital ring. Suppose that
each additive mapping $D:\mathcal{A}\rightarrow \mathcal{A}$
satisfying $(\ast)$ is a generalized derivation with $D(1)\in
Z(\mathcal{A})$. Let $T(\mathcal{A},\mathcal{A})$ be the trivial
extension of $\mathcal{A}$ by $\mathcal{A}$. Then every Jordan
derivation from $T(\mathcal{A},\mathcal{A})$ into itself is a
derivation.
\end{lem}
\begin{proof}
Let $\mathcal{T}=T(\mathcal{A},\mathcal{A})$ and
$\Delta:\mathcal{T}\rightarrow \mathcal{T}$ be a Jordan
derivation. We have
$\Delta((a,b))=(\delta_{1}(a)+\delta_{2}(b),\delta_{3}(a)+\delta_{4}(b))$
for each $a,b\in \mathcal{A}$, where
$\delta_{k}:\mathcal{A}\rightarrow \mathcal{A}$ ($k=1,2$) are
additive maps. Applying $\Delta$ to the equation
$(ab+ba,0)=(a,0)(b,0)+(b,0)(a,0)$ $(a,b\in \mathcal{A})$, we
deduce that $\delta_{1},\delta_{3}$ are Jordan derivations. Hence
$\delta_{1}$ and $\delta_{3}$ satisfy $(\ast)$ with
$\delta_{1}(1)=\delta_{3}(1)=0$. So by hypothesis $\delta_{1}$
and $\delta_{3}$ are derivations.
\par
Now by applying $\Delta$ to $$(0,a)(0,b)+(0,b)(0,a)=(0,0) \, \, \,
\textrm{and} \, \, \, (a,0)(0,b)+(0,b)(a,0)=(0,ab+ba)$$ for each
$a,b\in \mathcal{A}$, we get
\begin{equation}\label{16}
\delta_{2}(a)b+a\delta_{2}(b)+\delta_{2}(b)a+b\delta_{2}(a)=0,
\end{equation}
and
\begin{equation}\label{17}
\begin{split}
&\delta_{2}(ab+ba)=a\delta_{2}(b)+\delta_{2}(b)a,\\&
\delta_{4}(ab+ba)=\delta_{1}(a)b+a\delta_{4}(b)+b\delta_{1}(a)+\delta_{4}(b)a
\end{split}
\end{equation}
for all $a,b\in \mathcal{A}$. By Equation\eqref{17}, hypothesis
and Lemma~\ref{pj}, we get $\delta_{2}(a)=a\delta_{2}(1)$, for all
$a\in \mathcal{A}$, where $\delta_{2}(1)\in Z(\mathcal{A})$. Now
taking $b=1$ in Equation\eqref{16}, it follows that
$\delta_{2}(a)=-a\delta_{2}(1)$, for each $a\in \mathcal{A}$. So
$\delta_{2}(a)=0$ for all $a\in \mathcal{A}$. Define
$\varphi:\mathcal{A}\rightarrow \mathcal{A}$ by
$\varphi=\delta_{4}-\delta_{1}$. Then by Equation\eqref{17} we
get $\varphi(ab+ba)=a\varphi(b)+\varphi(b)a$ for all $a,b\in
\mathcal{A}$. Hence by Lemma~\ref{pj}, it follows that
$\varphi(a)=a\varphi(1)$ for all $a\in \mathcal{A}$, where
$\varphi(1)=\delta_{4}(1)\in Z(\mathcal{A})$ (since $\delta_{1}$
is a derivation, $\delta_{1}(1)=0$ ). Thus
$\delta_{4}(a)=\delta_{1}(a)+a\delta_{4}(1)$ for all $a\in
\mathcal{A}$, where $\delta_{4}(1)\in Z(\mathcal{A})$. By this
results it is obvious that $\Delta$ is a derivation.
\end{proof}
From Theorem~\ref{asli} and Lemma~\ref{ptri} we get the following
result.
\begin{thm}
Let $\mathcal{R}$ be a $2$-torsion free unital ring. Then every
Jordan derivation from $T(\Mn,\Mn)$ into itself is a derivation.
\end{thm}
\subsection*{Acknowledgment}
The author like to express his sincere thanks to the referees for
this paper.


\end{document}